\documentclass[leqno,11pt,a4paper]{amsart}

\usepackage{amsmath}
\usepackage{amsthm}
\usepackage{amsfonts,amssymb}
\usepackage{bbm}

\usepackage{latexsym}
\usepackage{mathrsfs}
\usepackage[all]{xy}
\usepackage{url}

\usepackage{float}
\usepackage{hyperref}
\usepackage{amssymb,longtable}
\usepackage{fontenc}

\usepackage{pdfsync}

\usepackage{amsthm}
\usepackage{hyperref}
\usepackage{enumerate}
\usepackage{url}

\usepackage{longtable}
\usepackage{booktabs}
\usepackage{colortbl}
\newcommand{\evnrow}{\rowcolor[gray]{0.95}}
\newcommand{\oddrow}{}

\usepackage[margin=2.5cm]{geometry}

\usepackage{pdflscape}
\usepackage{longtable}
\usepackage{booktabs}
\usepackage{colortbl}
\usepackage{arydshln}
\usepackage{verbatim}
\usepackage{enumitem}

\usepackage{mathdots}

\usepackage[english]{babel}
\usepackage{lipsum}
\usepackage{amsthm}
\usepackage{todonotes}
\usepackage[vcentermath]{youngtab}

\theoremstyle{plain}
\newtheorem{lema}{Lemma}[section]
\newtheorem{prop}[lema]{Proposition}
\newtheorem{thrm}[lema]{Theorem}

\theoremstyle{remark}

\newtheorem{rmk}[lema]{Remark}

\theoremstyle{definition}
\newtheorem{dfn}[lema]{Definition}

\def\QQ{{\mathbb Q}}
\def\CC{{\mathbb C}}
\def\NN{{\mathbb N}}
\def\AA{{\mathbb A}}

\def\PP{{\mathbb P}}
\def\wgr{\mathrm{wGr(2,5)}}
\def\Gr{\mathrm{Gr}}
\def\wp2{\mathrm{w}(\PP^2\times \PP^2)}
\def\wF{\mathrm{w}\mathcal{F}}
\def\wG{\mathrm{w}\mathcal{G}}
\def\wP{\mathrm{w}\mathcal{P}}

\def\Bs{{\mathcal{B}s}}

\def\w{{\mathrm{w}}}

\def\P{{\mathbf P}}

\def\Pf{{\mathrm{Pf}}}

\def\orb{\mathrm{orb}}
\def\GCD{\mathrm{GCD}}

 \newcommand{\into}{\hookrightarrow}

 \newcommand{\Oh}{\mathcal O}

 \newcommand{\sX}{\mathcal X}

 \newcommand{\Si}{\Sigma}
 
 \newcommand{\la}{\lambda}
 \newcommand{\al}{\alpha}

\newcommand{\PxP}{\PP^2 \times \PP^2}

\begin{document}

\author[M.I.~Qureshi]{Muhammad Imran Qureshi}
\address{Muhammad Imran Qureshi,
 Department of Mathematics, SBASSE
Lahore University of Management Sciences (LUMS)
Lahore, Pakistan and Mathematisches Institut, Universit\"at T\"ubingen, Germany}
\email{i.qureshi@maths.oxon.org}
\title{Biregular Models of log Del Pezzo surfaces with rigid singularities }


\begin{abstract}
We construct biregular  models  of  families of  log  Del Pezzo surfaces with rigid  cyclic  quotient singularities such that a general member in each family is wellformed and quasismooth. Each biregular model  consists of infinite series of such families of  surfaces; parameterized by the natural numbers $\NN$. Each family in these biregular models is  represented by either a codimension 3 Pfaffian format modelled on the Pl\"ucker embedding of $\Gr(2,5)$ or a codimension 4  format modelled on the Segre embedding of \(\PxP\).    In particular, we show the existence of two biregular  models in codimension 4 which are bi parameterized,  giving rise to an  infinite series of models of families of log Del Pezzo surfaces.  We identify those biregular models of surfaces which do not admit a  \(\QQ\)-Gorenstein deformation to a toric variety.\end{abstract}
\maketitle

\section{Introduction}
A  \emph{log  Del Pezzo surface} $X$ is a projective surface $X$ with $-K_X$ ample and having  isolated  cyclic quotient singularities; such surfaces are also referred to as  orbifold Del Pezzo surfaces with isolated orbifold points.
  The log Del Pezzo surfaces  form an interesting class of  surfaces which appear naturally in various contexts including the minimal model program \cite{Kawamata-MMP}.
   Recently, the construction and  classification of orbifold Del Pezzo surfaces have arisen in  the mirror symmetry program of Coates--Corti et al \cite{CCGGK} formally conjectured  in \cite{dp-AMS} for orbifold Del Pezzo surfaces. 
   Such  surfaces with $m\times \frac 13 (1,1)$ points have been classified by Corti and Heuberger in \cite{dp-CL}.  The classification with a single orbifold point of type  \(\frac 1r(1,1)\) is provided by Cavey and Prince \cite{dp-CP}. 
    The  log Del Pezzo surfaces have also been studied from the point of view of the existence of orbifold Kahler--Einstein metric on such surfaces in \cite{dp-zoo,dp-exceptional,Park} starting with \cite{Jhonson-Kollar}.
      In \cite{Jhonson-Kollar}, Johnson and Koll{\' a}r determined the complete list of del Pezzo hypersurfaces of index 1 in three-dimensional weighted projective spaces, admitting a Kahler-Einstein metric. In \cite{dp-zoo}, Cheltsov and Shramov classified the Del Pezzo hypersurfaces of index 1 in a weighted projective space  satisfying certain condition on their log-canonical threshold. 
      The wellformed and quasismooth weighted
complete intersection Del Pezzo surfaces  have been classified by Mayanskiy
in \cite{Mayanskiy}.
 
We construct biregular models  of   log Del Pezzo  surfaces  of Fano index \(I=1,2\)  in codimension  3 and 4 which are not complete intersections. Each biregular model consists of an infinite series of  families of wellformed and quasismooth  log Del Pezzo surfaces. Their equations  can be  described by maximal Pfaffians of a $5 \times 5$ skew symmetric matrix of forms or $2\times 2$ minors of the size 3 matrix   of forms.   We also compute their invariants like plurigenera,  degree of the canonical class etc. We identify those  families of surfaces which do not admit a $\QQ$-Gorenstein
deformation to a toric variety. Moreover, we also construct several  other families of wellformed and quasismooth  log Del Pezzo surfaces in codimension 3 and 4.

 Any normal surface $X$ with cyclic quotient singularities admits a $\QQ$-Gorenstein
 partial smoothing to a surface with only rigid singularities by \cite{Kollar-SB}.
 So  we only  concentrate on those surfaces having  quotient singularities which
are rigid under \(\QQ \)-Gorenstein smoothing. We give a  complete classification of such log Del Pezzo surfaces up to  certain values of a parameter called the  adjunction number of the free resolution of its corresponding graded ring, following \cite{BKZ,QJSC}. This is equivalent to finding all possible families  $\sX\subset \PP(a_i)$ with $\sum a_i \le q+I $, where $q$ is the adjunction number and $I$  is the Fano index of $\sX$.    It is not a complete list of such varieties which can be constructed using  these two  Gorenstein formats as the adjunction number \(q\) is unbounded. On the other hand the computational evidence suggest  that our list of biregular models is  complete. 
 
 We consider these families of log Del Pezzo surfaces as regular pullbacks (see \cite{BKZ}) from  key varieties of weighted Grassmannian \(\w\Gr(2,5)\) and \(\w(\PxP)\).  They also have a description as complete intersections in  some weighted Grassmannian $\wgr$  or  in the Segre embedding of weighted $\PxP$  or in the projective cone over either of those ambient  varieties, following \cite{QJSC}.  We exploit the latter description in the proofs. The most difficult part in our proofs is the singularity analysis of these families.  In  the case of orbifold Del Pezzo surfaces which are complete intersection in weighted projective space, the   quasismoothness can be proved by using criterion given by  Fletcher \cite{fletcher}.  But  we do not have a straight forward criteria in higher codimension and we have to prove it case by case. Thus the main challenging part of the computation is to  prove the quasismoothness of these  models.  The invariants like  $h^0(-K)$ and $-K^2$  can be calculated by using the Hilbert series  of these biregular models. The Hilbert series also helps us to identify those families which do not admit a $\QQ$-Gorenstein deformation to a toric variety.
    Our motivation is to  provide a vast testing ground for the questions like studying the  existence of an orbifold  Kahler--Einstein metric on a Del Pezzo surface or studying the mirror symmetry conjectures of \cite{dp-AMS}.  
\subsection*{Summary and Results}  \S\ref{S-Background} consists of the preliminary and background material for the proofs in the rest of the sections. We recall the definition of a $T$-singularity, a $R$-singularity and the  log Del Pezzo surface. We also introduce the notion of model of log Del Pezzo surfaces and  give a characterization for a quotient singularity $\frac 1r(a,b)$ to be an $R$-singulary in the context of this article. At the end, we give a general strategy of the proofs coming in the coming sections.  
 
In \S\ref{Pfaffians} we  construct biregular models of families of orbifold Del Pezzo surfaces in codimension 3 which are the  regular pullbacks from the key variety $\wgr$ which we refer to as Pfaffian models.  We recall the definition of weighted Grassmannian $\wgr$  and the formula for  Hilbert Series  from \cite{wg}. We describe how to compute the degree of the canonical class \(-K^2\) of a log Del Pezzo surface  appearing in Pfaffian models.   The main part of this section  gives the proof of Theorem \ref{thrm-Pf} where we show the existence, wellformedness and  quasismoothness of each model.      In total we construct eight biregular  models, four of them have Fano index 1 and four have   index 2. In fact, we
get 9 models but one of the index 2 models is not  quasismooth which we briefly discuss
in \ref{fail-Pf}.

\begin{thrm}  There are 8  biregular models (infinite series)  of  families of log Del Pezzo surfaces   with rigid singularities such that a general member of each family in each model  is  wellformed and  quasismooth  with the  basket of  singularities and invariants given in  Table \ref{tab-Pf}.  There equations  are given by  maximal  Pfaffians of the  \(  5 \times 5 \) skew symmetric matrix of forms giving the embedding of each family  $\sX\into \PP^5(a,\ldots,f)$. Moreover, at least 2 of these models do not admit a $\QQ$-Gorenstein deformation to a toric variety.
\label{thrm-Pf}  
 \end{thrm}
 \renewcommand*{\arraystretch}{1.2}
\begin{longtable}{>{\hspace{0.5em}}llccccr<{\hspace{0.5em}}}
\caption{Biregular log Del Pezzo Models in  Pfaffian format where  $q=r-1, s=r+1,t=r+2,u=r+3,v=2r+1,y=2r-2,z=2r-1,m=3r-2$. } \label{tab-Pf}\\
\toprule
\multicolumn{1}{c}{Model}&\multicolumn{1}{c}{WPS \& Param }&\multicolumn{1}{c}{Basket $\mathcal{B}$}&\multicolumn{1}{c}{$-K^2$}&\multicolumn{1}{c}{$h^0(-K)$}&\multicolumn{1}{c}{Weight Matrix}\\
\cmidrule(lr){1-1}\cmidrule(lr){2-2}\cmidrule(lr){3-3}\cmidrule(lr){4-5}\cmidrule(lr){6-6}
\endfirsthead
\multicolumn{7}{l}{\vspace{-0.25em}\scriptsize\emph{\tablename\ \thetable{} continued from previous page}}\\
\midrule
\endhead
\multicolumn{7}{r}{\scriptsize\emph{Continued on next page}}\\
\endfoot
\bottomrule
\endlastfoot
\evnrow $\Pf_{11}$ & 
  $\begin{matrix}
  \PP(1^3, r^2, z)\\  \; r \ge 2 \\ 
  w=\frac 12(1,1,1,z,z)
  \end{matrix}$&
    
$\dfrac{1}{z}(1,1)$
&  $\dfrac{2r+3}{2r-1} $             & 3&

$ \footnotesize{\begin{matrix} 1&1&r&r\\ &1&r&r\\ &&r&r\\ &&& z \end{matrix}}$  \\

\oddrow $\Pf_{12}$ & $\begin{matrix}
  \PP(1^2, 2,r^2,z)\\  r=2n+1,\; n\ge 1 \\
  w=(0,1,1,q,r)  
  
  \end{matrix}$ & $\dfrac{1}{r}(2,q),\dfrac{1}{z}(1,1)$&$\dfrac{2 r^2+5 r+1}{4 r^2-2 r}$&$  2 $&$ \footnotesize{\begin{matrix} 1&1&q&r\\ &2&r&s\\ &&r&s\\ &&& z \end{matrix}} $\\  

\evnrow $\Pf_{13} $&$ 
  \begin{matrix}
  \PP(1, 2,r^{2}, z,m),\\  r=2n+1,\; n \ge 1\\
 w= \frac 12(4-s,s-2,\\s,3s-6,3s-4)
  \end{matrix}$&$ 
    \begin{matrix}
 \dfrac{1}{r}(2,q), \dfrac{1}{m}(r,z)
\end{matrix}$&$\dfrac{3 r+1}{r (3 r-2)}$&$1
               $&$
 \footnotesize\begin{matrix} 1&2&r&s\\ &r&y&z\\ &&z&2r\\ &&& m \end{matrix} $ \\

\oddrow $\Pf_{14} $&$ 
  \begin{matrix}
  \PP( 2,r^{2},s,t, z)\\  r=2n+1,\; n \ge 1\\
  w=\frac 12(1,3,z,z,2r+1)
  \end{matrix}$&$     \begin{matrix}
  \dfrac{1}{t}(2,s),\dfrac{1}{z}(1,1)\\3\times \dfrac{1}{r}(2,q),
\end{matrix}$&$\dfrac{4 r+3}{r (r+2) (2 r-1)}
               $&$0$&$

 \footnotesize\begin{matrix} 2&r&r&s\\ &s&s&t\\ &&z&2r\\ &&& 2r \end{matrix}$  \\
 \evnrow $\Pf_{21} $&$ 
  \begin{matrix}
  \PP(1^2, 2,3,r, s);\\  r=3n,\; n \ge 2 \\
  w=(0,1,1,2,q)
  \end{matrix}$&$
    \begin{matrix}
\dfrac{1}{3}(1,1),\dfrac{1}{r}(1,1),\\\dfrac{1}{s}(3,r)
\end{matrix}$&$\dfrac{4 \left(2 r^2+4 r+3\right)}{3 r (r+1)}               $&$  2$&$

 \footnotesize{\begin{matrix} 1&1&2&q\\ &2&3&r\\ &&3&r\\ &&& s \end{matrix}}$ \\

\oddrow $\Pf_{22} $&$ \begin{matrix}
  \text{ rest same as } \Pf_{21}\\  r=3n+1,\; n \ge 2
  \end{matrix} $&$  \dfrac{1}{r}(1,1),\dfrac{1}{s}(3,r)$&$\text{ Same as } \Pf_{21} $&$ 2$&$ \text{Same as}  \Pf_{21}$\\  

\evnrow $\Pf_{23} $&$ 
  \begin{matrix}
  \PP(1, 3,r, s,t,u)\\  r=3n+2,\; n \ge 2\\
  w=(0,1,2,r,s)
  \end{matrix}$&$
    \begin{matrix}
 \dfrac{1}{3}(1,1),\\ \dfrac{1}{r}(1,1),\dfrac{1}{u}(3,t)
\end{matrix}$&$\dfrac{8 r+36}{3 r^2+9 r}
               $&$1$&$
 \footnotesize\begin{matrix} 1&2&r&s\\ &3&s&t\\ &&t&u\\ &&& v \end{matrix} $ \\

\oddrow $\Pf_{24} $&$ 
  \begin{matrix}
  \PP(3, r, s^2,t^2)\\  r=3n,\; n \ge 2\\
  w=\frac 12(q,s,s,u,u)
  \end{matrix}$&$
    \begin{matrix}
 \dfrac{1}{3}(1,1), \dfrac{1}{r}(1,1),\\\dfrac{1}{s}(3,r),2\times \dfrac{1}{t}(3,s)
\end{matrix}$&$\dfrac{4 (5 r+6)}{3 r \left(r^2+3 r+2\right)}$&$0
               $&$
 \footnotesize\begin{matrix} r&r&s&s\\ &s     &t&t\\ &&t&t\\ &&& u \end{matrix} $ \\
\end{longtable}
The first model in the above table has already been  discussed in \cite{dp-CP}  where its description as a toric variety has been provided. 

In \S\ref{Binomials} we prove the existence, wellformedness and quasismoothness  of some codimension 4  biregular models  of log Del Pezzo surfaces with rigid singularities. They can be realized as  regular pullbacks from  $\PxP$ format which we denote  by $\wP$.  We recall the definition of weighted $\PxP$ and a formula for its Hilbert series from \cite{Sz05}. We use the graded  free resolution information from its  Hilbert series to give a formula to compute the anti canonical degree  $-K^2$  of  log Del Pezzo surfaces in a given $\wP$ variety.  In total there are four biregular models of index 1. The case of index  2 models  is quite special  so we treat it separately.
 
\begin{thrm}  There are  4 biregular models  of  log Del Pezzo surfaces of Fano index 1 with   baskets of  rigid singularities in Table \ref{tab-P2xP2} such that a general member of each family in each model  is  wellformed and  quasismooth. Their equations are described by $2\times 2$ minors of the order 3 matrix of homogenous forms giving the embedding of each family   $\sX\into \PP^6(a,\ldots,g)$. At least, one of these biregular models does not admit a $\QQ$-Gorenstein deformation to a toric variety. 
\label{thrm-pxp1} 
 \end{thrm}
 \renewcommand*{\arraystretch}{1.2}
\begin{longtable}{>{\hspace{0.5em}}llccccr<{\hspace{0.5em}}}
\caption{biregular Log Del Pezzo Models in  $\PxP$ format where $q=r-1, s=r+1,t=r+2, z=2r-1$. } \label{tab-P2xP2}\\
\toprule
\multicolumn{1}{c}{Model}&\multicolumn{1}{c}{WPS \& Para\ }&\multicolumn{1}{c}{$\mathcal{B}$}&\multicolumn{1}{c}{$-K^2_X$}&\multicolumn{1}{c}{$h^0(-K)$}&\multicolumn{1}{c}{Weight Matrix}\\
\cmidrule(lr){1-1}\cmidrule(lr){2-2}\cmidrule(lr){3-3}\cmidrule(lr){4-5}\cmidrule(lr){6-6}
\endfirsthead
\multicolumn{7}{l}{\vspace{-0.25em}\scriptsize\emph{\tablename\ \thetable{} continued from previous page}}\\
\midrule
\endhead
\multicolumn{7}{r}{\scriptsize\emph{Continued on next page}}\\
\endfoot
\bottomrule
\endlastfoot

\evnrow $\P_{11} $&$ 
  \begin{matrix}
  \PP(1^4, r^2, z)\\  \; r \ge 2\\w=(0,0,q; 1,1,r)
  \end{matrix}$&$

\dfrac{1}{z}(1,1)
$&$ \dfrac{4 r+2}{1-2 r}              $&$ 4$&$
 \footnotesize{\begin{matrix} 1&1&r\\ 1&1&r\\ r&r& z \end{matrix}} $ \\

\oddrow $\P_{12} $&$ 
  \begin{matrix}
  \PP(1,2, r^2,s,t, z)\\  \; r=2n+1\; n\ge 1\\
  w=(0,1,q; 1,r,s)
  \end{matrix}$&$
\dfrac{1}{r}(2,q),\dfrac{1}{t}(2,s), \dfrac{1}{z}(1,1)
$&$\dfrac{2 r^2+7 r+1}{r (r+2) (2 r-1)}
 $&$ 1$&$  
 \footnotesize{\begin{matrix} 1&2&r\\ r&s&z\\ s&t& 2r \end{matrix}} $ \\

\evnrow $\P_{13} $&$ 
  \begin{matrix}
  \PP( 1,2^2,3,r^{2}, z)\\  r=2n+1,\; n \ge 1\\
  w=(0,1,q;1,2,r)
  \end{matrix}$&$
    \begin{matrix}
 \dfrac{1}{3}(1,1), \dfrac{1}{z}(1,1)\\
 2\times \dfrac{1}{r}(2,q)
 
\end{matrix}$&$  \dfrac{(2 r+3)}{3 r (2 r-1)}
              $&$1 $&$

 \footnotesize\begin{matrix} 1&2&r\\ 2&3&s\\ r&s&z\end{matrix} $ \\

\oddrow $\P_{14} $&$ 
  \begin{matrix}
  \PP( 2,3,r^{2},s,t, z)\\  r=6n-1,\; n \ge 1\\
  w=(0,1,q;2,r,s)
  \end{matrix}$&$

    \begin{matrix}
 \dfrac{1}{3}(1,1), 3\times \dfrac{1}{r}(2,q)\\
 \dfrac{1}{t}(3,r), \dfrac{1}{z}(1,1)
\end{matrix}$&$\dfrac{(r+1) (2 r+9)}{3 r (r+2) (2 r-1)}               $&$0$&$
 \footnotesize\begin{matrix} 2&3&s\\ r&s&z\\ s&t&2r\end{matrix} $ \\

\end{longtable}
The first model has been discussed in \cite{dp-CP} and its toric description has also  been provided by embedding it in a singular toric variety. In the  case of Fano index 2 we  get two biregular models which are indexed by 2  parameters giving rise to an infinite series  of one parameter biregular models in each case. More strikingly  up to adjunction number 68 we do not get a single
 quasismooth
family with rigid singularities  which is not in one of these  bi-parameterized  
models.         
       \begin{thrm}
       \label{thrm-pxp2}
 Let $(r,y) \in \NN \times \NN$ be a pair of positive integers having one the following types:   
\begin{itemize}
\item  $r=3m,\;y=3n,\;m,n \ge 2$
\item  $r=3m,\;y=3n+1,\;m,n \ge 2,\; $
\end{itemize}
Then  for each choice of input parameter $w=(0,1,q;1,2,y)$ we get a family of wellformed and quasismooth log Del Pezzo surfaces $$\sX \into \PP(1,2,3,r,s,y,z)$$ with  the weight matrix $$\left(\begin{matrix} 1&2&r\\ 2&3&s\\ y&z&w\end{matrix}\right) $$ where $q=r-1,s=r+1,z=y+1$ and $w=q+y$.   The degree of the canonical divisor  class in  terms of parameters $r$ and $y$ is given by 
$$-K_X^2= \dfrac{4 \left(r^2 (2 y+3)+r \left(2 y^2+4 y+3\right)+3 y (y+1)\right)}{3 ry (r+1)  (y+1)}.$$ The basket of orbifold points on  each family in both cases is given as follows. \end{thrm}

\begin{table}[H]
\renewcommand*{\arraystretch}{1.5}
\[
\begin{array}{c|c}
\toprule
\textrm{Model} &   \textrm{Basket}    \\
\cmidrule(lr){1-1}\cmidrule(lr){2-2}
\evnrow \P_{21} & 
 
    \begin{matrix}
\frac{1}{3}(1,1),\dfrac{1}{r}(1,1),\frac{1}{s}(3,r),\frac{1}{y}(1,1),\frac{1}{z}(3,y)
\end{matrix}\\

\oddrow \P_{22}  &
               
    \begin{matrix}
\frac{1}{r}(1,1),\frac{1}{s}(3,r),\frac{1}{y}(1,1),\frac{1}{z}(3,y)
\end{matrix}  \\

\bottomrule
\end{array}
\]
\end{table}
It is important to mention that the computer search gives another model with $r=3m+1$ and $y=3n$ but due to symmetry of the weight matrix along the diagonal it is isomorphic to $\P_{22}$.

In \S\ref{Sporadic}    we gave a summary of computational results obtained by using the computer search  in each of these two formats. The summary consists of a number of candidate families returned by computer, how many of those  contain only rigid singularities and how many of them are quasismooth.    We finish by  presenting the complete list of sporadic cases of  families of wellformed and quasismooth  log  Del Pezzo surfaces appearing in these formats,  up to certain adjunction number.
 
 \subsubsection*{Understanding Table \ref{tab-Pf} and \ref{tab-P2xP2}} The first column denotes the model name where $\Pf_{ij} $ ($\P_{ij}$) represents $j$-th model of index $i$. The second column contains the weights of the ambient weight projective space containing $X$ and two types of parameters: the first one tells us the form of the parameter $r$ and the second one gives the weights of the syzygy  matrix  appearing in the  last column of the table. The  3rd column contains the basket of singularities on  $X$. The column $-K^2$  represents  the degree of the canonical class \(-K_X\)  of $X$ in terms of parameter $r$ and $h^0(-K)$ contains the first plurigenus  of its Hilbert series. The last column contains the so called syzygy or weight matrix of $X$. 
 
\subsection*{Acknowledgement} I wish to thank     Alexander Kasprzyk   for some enlightening discussions which made me think about this project. I am also thankful to Gavin Brown and Yuri Prokhorov for helpful comments.  Thanks are also due to Nouman Zubair for setting up \textsc{Magma} on the HPC cluster of LUMS to run the computer calculations for this paper. Last but not least, I am grateful to an anonymous referee for helping me improve an earlier version of the  paper a great deal.   This research is supported by an Higher Education Commission (HEC)'s  NRPU  grant  "5906/Punjab/NRPU/RD/HEC/2016" and a fellowship of the Alexander--von--Humboldt foundation. 

 \section{ Preliminaries, Notation and General strategy of proofs}
 \label{S-Background}

  \subsection{Preliminaries}
Let $\mu_r$ denote the cyclic group   generated by a primitive $r$-th root of unity. Its acts on \(\AA^2_{x,y}\) by \(x\mapsto \epsilon^a x, y \mapsto \epsilon^b y  \). The quotient  is called a \emph{cyclic quotient} singularity or an orbifold point of type  $\frac{1}{r}(a,b)$. It is called isolated if  $r$ is relatively prime to  $a$ and $b$.  A  singularity is called a \emph{T-singularity} if  it admits
a $\QQ$-Gorenstein smoothing. It is called  an \emph{R-singularity} it it is is rigid under
any $\QQ$-Gorenstein smoothing. We  use the following  
characterization of a  cyclic quotient singularity to be a $T$-singularity and $R$-singularity, appeared in  \cite{dp-CP}.

\begin{dfn}
\label{dfn:T_R}
Given an arbitrary quotient singularity $Q = \frac{1}{r}(a,b)$, let  $m = \text{gcd}(a+b,r)$, $d = (a+b)/m$ and $k=r/m$. Then $Q$ can be written in the form $\frac{1}{mk}(1,md-1)$ and if :
\begin{enumerate}
\item  $k \mid m$ then $Q$ is  a  \emph{T}-singularity~\cite{Kollar-SB};
\item  $m<k$ then $Q$ is  a \emph{R}-singularity~\cite{AK-SC} .
\end{enumerate}
\end{dfn}
\noindent An algebraic surface is $\QQ$-Gorenstein if it is normal and the canonical divisor class is a \(\QQ \)-ample Weil divisor. 
A $\QQ$-Gorenstein algebraic surface    $X$ is a \emph{del~Pezzo surface} if the anti-canonical divisor class $-K_X$ is ample. If $X$ has at worst only  isolated quotient singularities then it is called \emph{orbifold or log Del Pezzo surface}.   The largest positive integer $I$ such that $-K_X = I\cdot D$ for some element $D $ in the divisor class group of $X$, is known as the \emph{Fano index} of a log Del~Pezzo surface $X$.
\begin{dfn}
A \emph{biregular model of log Del Pezzo surfaces} is  an infinite series of families of log  Del Pezzo surfaces satisfying the following conditions.  
\begin{enumerate}

\item  There exist a family of  log Del Pezzo surfaces for each value of the parameter $r(n)$ for all $n \in \NN$. 
\item Each family has the embedding  $\sX\into  \PP(w_i)$ such that at least one of the weights is $r$, and    each weight $w_i$  and the degree of the canonical class $(-K_\sX)^2$ are  functions of $r$.  

\end{enumerate}
\end{dfn}
We may only use the word ``model'' and  ``biregular model'' interchangeably if no confusion can arise. A biregular model is called \emph{wellformed and quasismooth model}  if a general member in each family is wellformed and quasismooth log  Del Pezzo surface. A surface $X\subset \PP(w_i) $ is quasismooth if the affine cone  $\widetilde{X}$ over $X$ is smooth outside the vertex  $\underline{0}$ and wellformed if at worst it contain the isolated orbifold points. We use the algorithmic approach of \cite{BKZ,QJSC} to search for the candidate biregular models which is primarily based on a theorem of  Buckley, Reid and Zhou \cite{BRZ}. The theorem gives a decomposition of the Hilbert series $P(t)$ of a projectively Gorenstein orbifold $X$  with isolated orbifold points into a smooth and orbifold part. The Gorenstein assumption on a surface $X$ with  $K_X=I\cdot D$  
 implies that  $\frac 1r (a,b)$  must satisfy \begin{equation}\label{PG}a+b+I= 0 \mod r.\end{equation} 
\begin{lema}Let $X$ be a log Del Pezzo surface of Fano index $1\le  I\le 2$, then the orbifold point $Q=\frac 1r(a,b)$ is a $T$-singularity if  it is either \(\frac 12(1,1) \text{ or }\frac 14(1,1)\). Otherwise, it is an $R$-singularity.
\end{lema}
\begin{proof}The proof follows from a straightforward application of  \ref{dfn:T_R}. If  \(m=\gcd(a+b,r)\) then for  $ I=1$  there are no orbifold point of type $\frac 12 (1,1)$ and $\frac 14(1,1)$   on $X,$ due to \eqref{PG}. Otherwise, we have $a+b=r+1$ thus $m=1<k=r$. If the Fano index \(I=2\) then if $r$ is odd then $m=1<k=r$.  Otherwise, if $r$ is even then $m=2$ and $k=r/2$. Now $k \mid m$ if $r=2,4$ otherwise  $m <k$.  
\end{proof}
In our proofs and calculations we repeatedly use the following Lemma to compute the exact number of singular points on orbifold loci of our varieties.
\begin{lema}\cite[Lemma 9.4]{fletcher}
\label{lemma-fletcher}
Let $X\subset \PP(a_0,a_1)$ be a general hypersurface of degree $d$ with $\gcd(a_{0},a_1)=1$. If $P_0 \text{ and }P_1$ denote the coordinate points $(1,0)$  and $(0,1)$ respectively, then $X$ is a finite set such that $P_i\in X$ if $a_j \nmid d$ for any $j=0,1$ and it contains \(\lfloor\frac d {a_0a_1}\rfloor\) further  points. 
\end{lema}

\subsection{Notation }
 
\begin{itemize}

\item We work over the field of complex numbers \(\CC\). We write $\sX$ for a family of Del Pezzo surfaces and $X$ for its general member.

\item Any isolated orbifold point $\frac 1r(a,b)$ can be written in the form
$\frac 1r(1,b')$ by using a different primitive generator of $\mu_r$, We
use the latter form of the orbifold points in all tables and examples.  We use the term quotient singularity and orbifold point interchangeably.
\item  All our orbifolds are projectively Gorenstein so each orbifold  point
of type $\frac 1r (1,b)$ has a presentation which satisfies the condition \eqref{PG}.
\item  We use $\wG$ to  denote the weighted Grassmannian $\wgr$ and $\wP$
will denote the ambient weighted $\PxP$ variety. 
\item $a,b,c,d,e,f$ and $g$ represent the   variables on the ambient
weight projective space containing a Del Pezzo surface $X$ whereas $m,r,s,t,u,v,y, \textrm{ and }z$
denote the weights of the variables depending on the model. The subscripts
will denote the degree of these variables in the proofs.  
\item The capital letters like  $H_ds \text{ and } J_ds$ (or only \((d)\) )  denote  homogeneous  forms
of degree $d$. 
\item We enclose the matrix of weights inside parentheses  $\left(\right)$   and matrix of variables and homogeneous forms inside square brackets $\left[\;\right]$ in the proofs. If we need to distinguish between two weights of same degree in the weight matrix then we  distinguish them with subscript, e.g. for example if we have two weights of degree $z$ then we denote them by $z_1, z_2$ in the weight matrix.

\end{itemize}  


\subsection{General strategy of the proofs}
\label{proof-strategy}
For each model, and the sporadic examples of \S\ref{Sporadic}, the proofs are divided into the following steps.
\subsubsection{Existence}The first  part is to show the existence of such models. We show the existence of such models by  constructing them as quasilinear sections of the given ambient key variety $\wgr$ or $w(\PxP)$ by specifying the choice of input parameters and quasilinear sections. The most important part of the existence is to show that each family of Del Pezzo orbifolds $\sX$ contains exactly those singular points which are suggested by the output in the computer search. A family may fail when it does not contain a suggested orbifold point, or sometimes when it contains 1-dimensional orbifold singularities.
\subsubsection{Wellformedness and Quasismoothness}  The wellformedness part is quite simple and follows straight away from the existence of models with right singularities in this case. If  $X\subset \PP(a_0,\cdots,a_n)$ is an orbifold Del Pezzo surface then
the orbifold singularities on $X$ occur due to the singularities of $\PP(a_0,\cdots,a_n)$.
 We start by computing the dimensions of each orbifold locus of $\PP(w_i)$ restricted to $X$ which
answers the question of wellformedness. Since we are on a surface,  $X$
is wellformed if and only if, at worst, it intersects in finite number of points with the singular strata of $\PP(w_i)$.

 The quasismoothness needs some detailed and careful analysis of two different types of  loci.
 One comes from the  singularities of ambient weighted projective space. The  second one may appear due to the base loci of the successive  linear systems of the intersecting weighted homogeneous forms. Outside of these loci, a general member in each family of these models  remain quasismooth due to the following version of Bertini's theorem.  
 \begin{thrm}[Bertini] 
 \label{Bertini}
 If a hypersurface  $ X \subset \PP(a_0,\ldots,a_n)$ is  a general element of a linear system  $L = |\Oh(d)|$, then the singularities (non-quasismooth points) of $X$
may only occur on the reduced part of the base locus of $L$. 
  \end{thrm}

If $p$ is an orbifold point of type $\frac{1}{r}(a,b)$ then it is either a coordinate point or lies on the  strata of dimension $\ge 1$. Usually $p$ is a coordinate point and    we show that  there exist $l$ weighted homogenous equations of the form 
$$F_k:=x_i^mx_k+\cdots $$
 in the ideal  of   $X$  where  $l$ is the codimension of $X$,  $x_i$ is the $i$-th coordinate point with weight $a_i:=r$ and the remaining two variables have weights $a$ and $b$ modulo $r$. These $l$ equations are called tangent polynomials \cite{BZ}. We call the variables $x_k$  the tangent variables  and rest of the variables   local variables near the point  $p$.  

Since we  consider the description of these  models  of log  Del Pezzo surfaces as complete intersections in some ambient $\w\Gr(2,5)$ or $\w(\PxP)$, i.e. a general memeber $X$ is of the form 
$$X=\left(\wG \text{ or }\wF\right) \cap \left(\cap (d_i)\right)\subset \PP(w_i).$$    There is a base locus of the linear  systems \(\left(|\Oh(d_i)|\right)\) of each form of degree $d_i$ where $X$ may have singularities  by Theorem \ref{Bertini}.  To prove that  $X$ is    quasismoothness on  the base locus, 
in few cases we use a purely theoretical arguments and mostly, a combination of theoretical and computational evidence. A theoretical arguments works well if $X$ intersect with the base loci in finite number of points.   Otherwise the base loci is very complicated and it becomes very difficult to show the quasismoothness theoretically. In such cases, we  show  that the base locus  remains geometrically does not change for each value of the parameter $r$.  We give a detailed proof in one of the cases, i.e. in section \ref{S-Pf12} and  rest of the cases are similar. Thus it suffices to show that $X$ is quasismooth  for first few values  of $r$, to establish quasismoothness for a given model.  For small values of $r$, we show the quasismoothness by using the computer algebra system $\textsc{Magma}$ by writing down its equations over the rational numbers. To provide further evidence, we verify it for    as large values of $r$  as computationally feasible by computer algebra. We write down the largest value of parameter which we checked by computer algebra, in each case separately in the proofs.  
   \subsubsection{Non existence of deformations to a toric variety}Each family in our  models of orbifold Del Pezzo surfaces is locally $\QQ$-Gorenstein rigid. From \cite{dp-AMS}, we know that   $h^0(-K)$ (more generally all  plurigenera $h^0(-mK))  $  is invariant under $\QQ$-Gorenstein deformations. Now if $X_\Si$ is a toric Fano variety then $h^0(-K)>0$ since the  origin is always contained in the corresponding Fano polytope. Thus if each family in our model has  $h^0(-K)=0$ then the we say that the  model does not admit a $\QQ$-Gorenstein deformation to a toric variety.  Such deformation families are more interesting in a sense that they can not be constructed using toric methods.  In fact two of our models do admit   \(\QQ\)-Gorenstein deformation to  toric varieties, these appeared in  \cite{dp-CP}. But  we do not treat this question in this paper for all our orbifold models.

\section{Pfaffian models}
\label{Pfaffians}
\subsection{Generalities on $\wgr$}\label{basic-pfaffians}This part mainly consists of the selected material from \cite{wg} where the detailed treatment of the subject can be found.  In the rest of this section we denote it by \(\wG\).
\begin{dfn}
 Consider a 5 tuple of  all integers or  half integers $w:=(w_1,\cdots,w_5)$   such that $$w_i+w_j>0,\; 1\le i<j\le 5,$$ Then the $\wG$ is the quotient of punctured affine cone $\widetilde{\Gr \backslash\{\ 0\}} $ by $\CC^\times$: 
   $$\la:x_{ij}\mapsto \la^{w_i +w_j }x_{ij}$$ 
  where \(x_{ij}\) are Pl\"ucker coordinates of the embedding \(\Gr(2,5)\into \PP\left(\bigwedge^2 \CC^5\right).\) Thus we get the
embedding 
$$\wG\into \PP(\{w_i+w_j: 1\le i <j\le 5\}),  $$which is defined by the  5 maximal  Pfaffians of $5\times 5$ skew symmetric.\end{dfn} We refer to this skew symmetric matrix as weight matrix and usually present by only writing down the upper triangular part, 
   $$\left(\begin{matrix} 
w_{12}&w_{13}&w_{14}&w_{15}\\ 
&w_{23}&w_{24}&w_{25}\\ 
&&w_{34}&w_{35}\\ 
&&& w_{45} \end{matrix}\right)
,$$
where $w_{ij}=w_i+w_j$.
The Hilbert series of $\wG$ is given by $$P_{\wG}=\dfrac{ 1-\displaystyle\sum_{i=1}^5t^{d-w_i}+\sum_{i=1}^5t^{d+w_i}-t^{2d}}{\displaystyle\prod_{i,j}(1-t^{w_i+w_j})},$$
where $d=\sum w_i$. If $\wG$ is wellformed, the orbifold canonical canonical class is $$K_\wG=\Oh(2d-\sum_{i,j} w_i+w_j)=\Oh(-2d).$$  The degree of  weighted Grassmannian $\wG$ is 
\begin{equation}
\deg\wG\ =\ \dfrac{\displaystyle\binom {2d}3-\sum_{i=1}^5\binom{d+w_i}3+\sum_{i=1}^5\binom{d-w_i}3}
 {\prod (w_i+w_j)}. 
\end{equation}
 
 Then obviously if $X=\wG \cap \left(\cap_{i=1}^4
(f_i)\right)$ is 
complete intersection Del Pezzo orbifold of index $I$ then \begin{equation}-K_X^2=I^2 \prod_{i=1}^4\deg(f_i)\deg(
\wG)  \end{equation}
The degree of each model has been computed using the above formula by using the  computer algebra software {\tt Mathematica}.


\subsection{Proof of Theorem \ref{thrm-Pf}}
We first give the proof of existence of each model with right invariants and singularities. Then we prove the  quasismoothness of the each model. Here we are taking the point of view appeared \cite{QJSC,qs-ahep,QJGP}, i.e. to consider each $X$ as weighted complete intersection of some  $\wG$. Indeed, it is a a special case of being consider  them as regular pullbacks of some key variety, like in  \cite{BKZ}. We write the proof for a general member  $X$ of a family $\sX\subset \PP(a,b,c,d,e,f)$ of orbifold Del Pezzo surfaces. In the course of the proof our parameter is $r$ and the rest of the weights in terms of $r$ are  $$\begin{array}{llll}q=r-1,& s=r+1,&t=r+2,&u=r+3,\\v=2r+1,&y=2r-2,&z=2r-1,&m=3r-2.\end{array}$$
\subsubsection{\textbf{Pfaffian Model 11}} 
This is the  simplest of Pfaffian models and quite straight forward to prove each family is  quasi sooth. The parameter is  $r=n, n \in \NN $ and   if we choose an input parameter \(w=\frac 12(1,1,1,z,z)\) we get the embedding of 6 dimensional orbifold  $$\wG \into \PP(1^3,r^6,z) \text{ with  the orbifold
canonical class } K_\wG=\Oh(-4r-1).$$  Then if we take the complete intersection of $\wG$  with  4 forms of degree $r$ then
$$X:= \wG\cap (r)^4 \into \PP(1^3,r^2,z)=\PP(a,b,c,d,e,f)$$ is a   Del Pezzo surfaces of index 1.  \\
\textit{Quasismoothness:} $ X$ has two non-trivial orbifold strata, one is of weight $z$ and the other of weight $r$. The weight $z$ locus is just a point which obviously lies on $X$. The variables $a, b,c$ serve as tangent variables and $d,e$ as local variables near  to  $f\ne 0$. Thus it is a point of type $\frac 1z(r,r)=\frac{1}{z}(1,1)$. On the other hand the weight $r$ locus is a copy of the Segre 3-fold  $\PP^1 \times \PP^2$ in $\wG$  which does not intersect $X$. \\
 The  base locus of linear system $|\Oh(r)|$ consists of just a coordinate point of weight $z$: $$\Bs\left(|\Oh(r)|\right)=(0,0,0,0,0,1),$$ 
which is  quasismooth. Thus the each member in this models  is a wellformed and quasismooth family of log Del Pezzo surfaces.    
\subsubsection{\textbf{Pfaffian model 12}} 
\label{S-Pf12}
The paramter has the form $r=2n+1, n \in \NN$  and   for an input parameter \(w=(0,1,1,q,r)\), we get   $$\wG \into \PP(1^2,2,q,r^3,s^2,z) \text{ with  } K_{\wG}=\Oh(-2(2r+1)).$$   If we take the complete intersection  of $\wG $ with two forms of degree $s$ and one form each of degree   $r$ and $q$ then 
$$X:= \wG\cap (s)^2\cap (q)\cap (r) \into \PP(1^2,2,r^2,z)=\PP(a,b,c,d,e,f)$$ is a  log Del Pezzo surfaces of Fano index 1. The equations are given by the $4\times 4 $ Pfaffians of skew symmetric matrix $$\left[\begin{matrix} a_1&b_1&H_q&d_{r_1}\\ &c_2&H_{r_2}&H_{s_1}\\ &&e_{r_3}&J_{s_2}\\ &&& f_z \end{matrix}\right],$$ where \(H\) and \(J\) denote the general form of degree  and the subscripts denotes the weights of variables and general forms.\\
\textit{Quasismoothness:} Now $ X$ has  3 different orbifold loci, having weight $z$, $r$ and 2. The weight $z$ locus is again just  a point which obviously lies on $X$. It is obviously a point of type  $\frac{1}{z}(1,1)$. 

The weight  $r$ locus can be taken as a coordinate point   $(0,0,0,0,1,0)$ of  the variable $e$. Let $$H_s=be+\cdots$$  then $a,b$ and $d$ are tangent variables. Then we get $ c $ and $f$ are local variables near this point which gives an orbifold point of type $\frac{1}{r}(2,q)$.

The weight 2 locus on $X$ is $V(cf,cH_q)\subset \PP(2,z)$ which is an empty set. 
\\
Now  we analyze the base loci of each linear system of weighted homogenous forms. We start with $$X_1=\wG\cap (s)^2\subset \PP(1^2,2,q,r^3,z);$$  the intersection of  $\wG$ with two forms of weight $s$. Then the base locus of the linear system of degree $s$ restricted to $X_1$  is $$\Bs\left(|\Oh(s)|\right)\cap X_1=\PP[q,r_1,r_2,z]\cup \PP[q,r_3,z].$$ Then  $X_1$ is quasismooth away from this locus. If we take  $X_2=X_1\cap (q)$ and then we have  $$\Bs(|\Oh(q)|)\cap X_2=\PP[r_1,r_2,z]\cup \PP[r_3,z].$$ We do not get any new base locus in this case and  $X_2$ is quasismooth away from this locus. At the end, we  get  $X=X_2\cap (r_2)$ which intersects with the base locus of $|\Oh(r_2)|$ in   two coordinate points of  $\PP[2,z]$ which are manifestly quasismooth.  Moreover, taking last two intersections reduces  $\Bs(|\Oh(q)|)$  to $\PP[r_2,z] \cup \PP[r_3,z]$. Thus $X$ is quasismooth outside this locus. The loucs geometrically does not change for all values of the parameter and we use the computer algebra to prove the quasismoothness of this model. The following \textsc{Magma} function shows the quasismoothness for any value of the parameter $r$. 

\begin{verbatim}
function Pf12(r)
rpoly := func< P,d | d ge 0 select
&+[ Random([1..7])*m : m in MonomialsOfWeightedDegree(CoordinateRing(P),d)]
else CoordinateRing (P)!0 >;
P<x12,x13,x23,x15,x34,x45>:=ProjectiveSpace(Rationals(),[1,1,2,r,r,2*r-1]);
f4:=rpoly(P,r-1);f5:=rpoly(P,r);f6:=rpoly(P,r+1);g6:=rpoly(P,r+1);

M := -AntisymmetricMatrix([
x12,
x13, x23,
f4, f5, x34,
x15, f6, g6, x45 ]);

X := Scheme(P,Pfaffians(M,4));
SX := JacobianSubrankScheme(X);
SXred := ReducedSubscheme(SX);
D:=Dimension(SXred);
return D;
end function;
\end{verbatim} 
In this model, we checked the quasismoothness for $3\le r \le 49$. 
 \subsubsection{\textbf{Pfaffian model 13}} 
\label{S-Pf13}
The parameter has the form $r=2n+1, n \in \NN$  and if we choose an input  parameter \[w=\frac 12(4-s,s-2,s,3s-6,3s-4)\] we get an embedding of   $$\wG \into \PP(1,2,r^2,s,y,z^2,2r,m) \text{ with }K_\wG=\Oh(-(7r-1)).$$ If we take the complete intersection  of $\wG $ with four forms having degrees  $z,2r,y$ and    $r$, we get a  log  Del Pezzo surface 
$$X:= \wG\cap (z)\cap(2r)\cap (y)\cap (s) \into \PP(1,2,r^2,z,m)=\PP(a,b,c,d,e,f).$$  The equations can be described by the  maximal Pfaffians of the  skew symmetric matrix 
\renewcommand*{\arraystretch}{1.1}
$$\left[\begin{matrix} 
a_1&b_2&c_r&H_s\\ 
&d_r&H_y&H_z\\ 
&&e_z&H_{2r}\\ 
&&& f_m \end{matrix}\right].$$
\textit{Quasismoothness:} We show that   $ X$ has   only four  distinct orbifold loci with weights  $m,z,r$ and 2.\\
 The weight $m$ locus is just the coordinate  point of the variable $f$ which lies on $X$. The variables  $a,b$ and $d$ can be removed by using the implicit function theorem near this point and $c,e$ are local variables near this point. Thus we get an orbifold point of type $\frac{1}{m}(r,z)$. This indeed represents a point as  $m$ is relatively prime to both $r$ and $z$.\\
  The weight $z$ locus is just an empty set as we have a term  $e^2$ in one of the defining equations of $X$.
 The equations of $X$ restricted to  weight $r$ variables is given by 

$$V(cH_{2r},cd)\subset \PP(c,d) $$ 
which is manifestly the coordinate point of variable $d$ on $X$. If $H_s=ad+\cdots$ then $a,c$ and $f$ are tangent variables and we get an orbifold point of type $\frac{1}{r}(2,q)$ on $X$.\\
The weight $2$ locus consists of intersection of  $X$ with $\PP(2,m)$ as $m$ is even. In the equations $bf \text{ and }bH_y$; the variable  $f$ does not appear in  $H_y$ due to the reason of degree so it can be at most a coordinate point of   variable $m$ which we already considered earlier. Thus we do not get any  $\frac 12$ type of singular points on \(X\).\\ 
In this case we can also show that the base locus geomtrically remains constant for any value of $r$, like in \ref{S-Pf12}.  Thus we used a computer algebra calculation and verified the quasismoothness for  all $ 3\le r\le 199 $.  
\subsubsection{\textbf{Pfaffian model 14}} 
\label{S-Pf14}
The paramter has the form $r=2n+1, n \in \NN$  and for an input  parameter \(w=\frac 12(1,3,z,z,2r+1)\) we get an embedding of   $$\wG \into \PP(2,r^2,s^3,t,z,(2r)^2) \text{ with } K_{\wG}=\Oh(-3(2r+1)).$$ Then a Del Pezzo surface  \(X\) of index 1 is a complete   intersection of  $\wG $ with two forms each of degree   $2r$ and  $s$. 
$$X:= \wG\cap (2r)^{2}\cap (s)^{2} \into \PP(2,r^2,s,t,z)=\PP(a,b,c,d,e,f).$$  The equations can be described by the  maximal Pfaffians of the  skew symmetric matrix 
$$\left[\begin{matrix} 
a_2&b_r&c_r&d_s\\ 
&H_s&J_s&e_t\\ 
&&f_z&H_{2r}\\ 
&&& J_{2r} \end{matrix}\right].$$
\textit{Quasismoothness:} We show that each  $ X$ has   only three  types of singularities with weights  $r,t$ and $z$  and no other orbifold singularities.\\
 The weight $z$ locus is  just a coordinate point and   $a,d$ and $e$ serve as  tangent variables. Thus we get   $b \text{ and } c$ as local variables near this point and we get an orbifold point of type $\frac{1}{z}(1,1)$.\\
 The weight $t$ locus is again just a coordinate point. In this case, $b,c \text{ and } f $ are tangent variables so it is an orbifold point of type $ \frac{1}{t}(2,s)$. Since $t$ is odd, this is only isolated singular point.\\ $X$ restricted to weight $s$ locus is  an empty set as we have a pure power of $d$ appearing in the equations of $X$. 
 The equations of $X$ restricted to  weight $r$ variables is manifestly a cubic in $\PP^1$ given by 

$$V(bJ_{2r}-cH_{2r})\subset \PP(b,c). $$ 
   If $H_{2r}=cb+\cdots, \text{ and }J_{2r}=bc+\cdots,$ then on each affine piece we can easily show that the local variables are of weight $2$ and $ z$ which gives us the 3 points of type  $\frac{1}{r}(2,q)$ on $X$.\\
At the end 
the weight $2$ locus,  $X \cap \PP(2,s)$ as $s$ is even, does not intersect $X$.\\
Thus each $ X$ contains correct type of orbifold singularities. \\ 
 In this case we used computer algebra to verify the quasismoothness for  $ 3 \le r \le 199, 1161\le r\le 1199  $  and $11161 \le  r \le 11199$. Three different ranges were chosen to further assert the verification of quasismoothness. 
\subsection*{Index 2 models}
The proof for index 2 cases are very similar to index 1 models. Therefore, we will give short summary of the quasismoothness on the orbifold locus  in a tabular form  to illuminate all the properties of the proof.  We list the details of each orbifold loci in a small tabular form by writing down the tangent  variables and the local variables in the neighbourhood of the corresponding open affine patches. Moreover, we also write the conditions needed on the intersecting weighted homogeneous forms to find all the tangent variables in each case. We will denote the weight of the singular strata under consideration  with $P_\orb$.

\subsubsection{\textbf{Pfaffian Model 21}} 
\label{S-Pf21}
The paramter has the form $r=3n, n \in \NN$  and  for the input parameter $w=(0,1,1,2,q)$
we get the ambient weighted projective variety
 $$
 \wG \into \PP(1^2,2^2,3^2,q,r^2,s) \text{ with }
K_{\wG}=\Oh(-2(r+3)).$$
Then the complete intersection  of $\wG$  with four forms of degree   $r,q,3$  and $2$  is a log Del Pezzo surface 
$$X:= \wG\cap (r)\cap(q)\cap (3)\cap (2) \into \PP(1^2,2,3,r,s)=\PP(a,b,c,d,e,f)$$ of index 2.   The equations are given  by the  maximal Pfaffians  
$$\left[\begin{matrix} 
a_1&b_1&c_2&H_q\\ 
&H_2&d_3&e_r\\ 
&&H_3&H_{r}\\ 
&&& f_{s} \end{matrix}\right].
$$

\textit{Quasismoothness:}
We summarise the  details of the orbifold loci as follows.  
$$
\renewcommand*{\arraystretch}{1.2}
\begin{array}{cccl}
P_\orb & X\cap P_\orb  & \text {Tangent}|\text{local variables}  & \text{Conditions on forms}  \\\hline
        s & \text {coordinate pt } f_s & a,b,c\;|\;d,e & H_2=c+\cdots \\
                r & \text {coordinate pt } e_r & b,c,d\;|\;a,f & H_3=d+\cdots \\
                                3 & \text {coordinate pt-} d_3 & b,e,c\;|\;a,f & H_q=d^nc+\cdots \\

        \end{array}
$$
 Since $r$ is a multiple of $3$, the  weight 3 orbifold locus is $$V(dH_r-e H_3)\subset \PP[3,r].$$ This is manifestly 2 points by using Lemma \ref{lemma-fletcher}. The one new point is the coordinate point $d$ of weight $3$. The other one already appeared as weight $r$ orbifold point. The weight 2 locus is $\PP[2,r]$ if $r$ is even and and $\PP[2,s]$ otherwise. In both cases it does not intersect $X$ as $H_2=c+\cdots$.\\ 
 Thus  $X$ is wellformed and quasismooth on the orbifold locus. In this model, we use the computer algebra to checked that $X$ is quasismooth. We verify it for \(6\le r\le 69  \) that it is quasismooth. Since the base locus remains the same for all $n$, we conclude that  $X$ is quasismooth for all values of $n$.   
 \subsubsection{\textbf{Pfaffian Model 22}} 
 \label{S-Pf22}
 This case is exactly similar to the \ref{S-Pf21} albeit our parameter is $r=3n+1$. The proof of quasismoothness is only different at weight $3$ and $2$ orbifold loci. The  weight $3$ orbifold locus is  just the coordinate point  of variable $d$, which does not lie on \(X\) as $H_q=d^n+\cdots$. Similarly the weight 2 locus is $\PP[2,r]$ if $r$ is odd  and $\PP[2,s]$ otherwise. In both cases it does not intersect $X$.\\ 
 Thus  $X$ is wellformed and quasismooth on the orbifold locus and has the right type of singularities. Just like the last case,  we used the computer algebra to verify that $X$ is quasismooth which gives a  $ 7\le r\le 70$. Since the base loci remains the same for all $r$, we conclude that this  model is quasismooth for all values of $r$.  
\subsubsection{\textbf{Pfaffian Model 23}} 
\label{S-Pf23}
The paramter has the form $r=3n+2, n \in \NN$  and for a choice of input parameter $w=(0,1,1,r,s)$
we get the ambient weighted projective variety
 $$
 \wG \into \PP(1,2,3,r,s^2,t^2,u,v)
\text{ with }K_\wG=\Oh(-4(r+2)).$$
We take the  complete intersection  of $\wG$  with four forms of degrees   $v,t,s$  and $2$, to  get  a log  Del Pezzo surface 
$$X:= \wG\cap (v)\cap(t)\cap (s)\cap (2) \into \PP(1,3,r,s,t,u)=\PP(a,b,c,d,e,f)$$ of  index 2.   The equations are given  by the  maximal Pfaffians of the  skew symmetric matrix 
$$\left[\begin{matrix} 
a_1&H_2&b_r&H_s\\ 
&c_3&d_s&H_t\\ 
&&e_t&f_{u}\\ 
&&& H_{v} \end{matrix}\right].
$$

\textit{Quasismoothness:}
The details of weight $s$ and weight $r$ orbifold point is given in the following small table.
$$
\begin{array}{cccl}
P_\orb & X\cap P_\orb \ & \text {Tangent}|\text{local variables}  & \text{Conditions}  \\
        
                u & \text {coordinate pt } f_u & a,b,d\;|\;c,e &  \\
                r & \text {coordinate pt } b_r & c,e,f\;|\;a,d & H_t=e+\cdots \\
                                3 & \text {coordinate pt } c_3 & b,d,f\;|\;a,e & H_v=c^nf+\cdots \\

        \end{array}
$$
 The orbifold strata of weight $t$ and $s$  misses $X$ as we have the pure powers  of $e$ and $d$ in the equations of $X$ if $$H_t=e+\cdots \text{ and } H_s=d+\cdots.$$ The weight 3 locus is $X\cap\PP[3,s];$  explicitly $$V(dH_s,cd)\subset \PP[c,d]$$ which  is the coordinate point of the variable $c$.  
Even though there is no variable of weight  2 in the ambient \(\w\PP\) but   $$\GCD(r,t)=2;\text{ if } r \text{ is even and }  \GCD(t,u)=2; \text{ if } r \text{ is odd  },$$  we have to calculate orbifold locus of weight $2$ which may appear on $x$. In both cases we either get coordinate point of weight $r$ or $u$ which we already accounted for, so we do not get any new orbifold point. 
For the quasismoothness on the base locus, we instead use the computer algebra system to show that $X$ is quasismooth for  $5\le r\le 299$. Since base loci remains the same, we conclude that $X$ is quasismooth for all $n$.  
 
\subsubsection{\textbf{Pfaffian Model 24}}
\label{S-Pf24}
The paramter has the form $r=3n, n \in \NN$  and for a choice of input parameter $w=\frac 12(q,s,s,u,u)$
we get the ambient weighted projective variety
 $$
 \wG \into \PP(r^2,s^3,t^4,u) \text{ with } K_\wG=\Oh(-(5r+7)).$$
We take a projective  cone of weight 3 to get a 7-fold $\mathcal{C}^3\wG$ with the canonical divisor class $\Oh(-5(r+2)).$ Then we take the intersection of this 7-fold with  two forms of  weight $t$ and one form each of weight $u,s,r$ to get  a  Del Pezzo orbifold of index 2 
$$X:= \mathcal{C}^3\wG\cap (t)^2\cap(u)\cap (s)\cap (r) \into \PP(3,r,s^2,t^2)=\PP(f, a,b,c,d,e).$$    The equations are given  by the  maximal Pfaffians of the  skew symmetric matrix 
$$\left[\begin{matrix} 
a_r&H_r&b_s&H_s\\ 
&c_s&d_t&e_t\\ 
&&H_t&J_{t}\\ 
&&& H_{u} \end{matrix}\right].
$$
\textit{Quasismoothness:}
The details of orbifold loci lying on $X$ is summarized as follows.

$$
\renewcommand*{\arraystretch}{1}
\begin{array}{c|cc|l}
P_\orb & X\cap P_\orb \ & \text {Tangent}|\text{local variables}  & \text{Conditions on forms}  \\  
\evnrow t&  \begin{matrix}
\text { 2 points by \ref{lemma-fletcher} }\\
\text{ on patch } d\\
\text{ on patch }e

\end{matrix}   & \begin{matrix}
\\
a,b,e\;|\;c,f \\
a,c,d\;|\;b,f
\end{matrix}&\begin{matrix}\\
H_r=a+\cdots,H_3=b+\cdots\\
H_t=d+\cdots,
J_t=e+\cdots
\end{matrix}  \\
                s & \text {coordinate pt } b_s & c,e,d\;|\;a,f & J_t=d+\cdots \\
\evnrow                         r & \text {coordinate pt } a_r & d,e,f\;|\;b,c & H_u=af+\cdots,H_t,J_t \text{ as for weight } t \\
                3 & \text {one new coordinate pt } f_3 & d,e,a\;|\;b,c & \begin{matrix}
H_r=f^n+\cdots, H_u=f^{n+1}a+\cdots             
\end{matrix}             
        \end{array}
$$
 The weight $t$ locus restricted to $X$ consists of 2 points. On each affine patch we get local variables of weight 3 and $s$ modulo $t$. The weight 3 locus is restricted to $X$ is given by  $$V(aH_u,H_rH_u)\subset \PP(3,r),$$ giving 2 coordinate points of weight $3$ and $r$. At the end, we may get singularities of weight 2. If $r$ is even then the weight 2 locus is 
 $$V(dJ_t-eH_t,aH_t-dH_r,aJ_t-eH_r)\subset \PP[r,t^2],$$
 which gives 3 points. But there is no new singularity as we already got 2 points of type $\frac{1}{t}(3,s)$ and a point of type $\frac{1}{r}(1,1)$. If $r$ is odd, then it is a coordinate point of weight $s$ which we already accounted for and we do not get any new singularities. In this case the compute algebra computations run very fast and we verified the quasismoothness for $6\le r \le 30,000. $
\begin{rmk}
\label{fail-Pf}
It is important to mention that  the numerical candidate examples or models do not always give rise to  quasismooth model of Del Pezzo surface. In the case of Fano index 2, if we use the   same numerical data  as in the  model $\Pf_{23}$ in  \ref{S-Pf23} 
for $r=3n+1.$ It appears to be an  another  model of families of orbifold Del Pezzo
surfaces. It  satisfies all the properties of the suggested model except quasismoothness  on one affine patch of the weight 3 locus.  We can not find 3 tangent monomials on  the
affine patch of weight $t$, so we do not include this model in our lists. 
\end{rmk}
\section{$ \PP^2 \times \PP^2$ models}
\label{Binomials}
\subsection{Generalities on $\w(\PxP)$}
We first recall the definition of weighted $\PxP$ and formula for its Hilbert Series from \cite{BKQ-p2xp2,Sz05} which we denote by $\wP$ for the rest of this article. 
\begin{dfn} For a choice of  two integer or half integer vectors $a=(a_1,a_2,a_3)$ and $b=(b_1,b_2,b_3)$ which satisfy  $$a_1+b_1 > 0,a_i\le a_j \text{ and } b_i\le b_j \text{ for } 1\le i\le j \le 3,$$we define a $\wP$ as the quotient by $\CC^\times$ of the punctured affine cone $\widetilde{\mathcal P\backslash\{0\}}$ of the Segre embedding $\mathcal {P}=\PxP\into \PP^8 $ by 
$$\la:x_{ij}\mapsto\la^{a_i+b_j}x_{ij}, \; 1\le i,j\le 3.$$
where the $x_{ij}$ are the coordinates of $\PP^8$. Thus we get the embedding of $$\wP\into \PP^8(a_1+b_1,\dots,a_3+b_3) $$ for the choice of $a,b$, written together as a single  input parameter $w=(a_1,a_2,a_3;b_1,b_2,b_3)$. The equations are defined by $2\times 2$ minors of the $3\times 3$ matrix which we usually refer to as  weighted matrix and write it as follows.
\begin{equation}\label{eq!wtmx}
\begin{pmatrix} x_{11} & x_{12} & x_{13} \\ x_{21} & x_{22} & x_{23} \\ x_{31} & x_{32} & x_{33} \end{pmatrix}
=
\begin{pmatrix} a_1+b_1 & a_1+b_2 & a_1+b_3 \\ a_2+b_1 & a_2+b_2 & a_2+b_3
\\ a_3+b_1 & a_3+b_2 & a_3+b_3  \end{pmatrix}
\end{equation}
\end{dfn}
 The Hilbert series of $\wP$ is given by $$P_{\wP}(t)=\dfrac{1-\displaystyle\left(\sum_{1\le
i,j \le 3}t^{-\al_{ij}}\right)t^{d}+\left(4+\sum_{1 \le i\ne j\le 3}t^{\al_{ji}}+\sum_{1\le
i\ne j\le 3}t^{\beta_{ji}}\right)t^{d} -\left(\sum_{i,j}t^{\al_{ij}}\right)t^{d}+t^{2d}}{\displaystyle\prod_{1\le
i,j\le 3}\left(1-t^{a_i+b_j}\right)},$$ where 
$d=a_1+a_2+a_3+b_1+b_2+b_3,\al_{ij}=a_i+b_j, \al_{ji}=a_i-a_j$ and $\beta_{ji}=b_i-b_j.  $ If $\wP$ is wellformed then the orbifold canonical class is given by $$K_{\wP}=\Oh_{\wP}(2d-\sum_{i,j}a_i+b_j)=\Oh_{\wP}(-d).$$
\begin{prop} The degree of weighted $\PxP$ variety is given by 
\begin{equation}
 \deg (\wP)=\dfrac{\displaystyle\binom {2d}4 + 4\binom d 4+\sum_{1\le i\ne j \le 3}\left(\binom{d+\beta_{ji}}4+\binom{d+\al_{ji}}4\right)-\displaystyle\sum_{1\le i,j\le
3}\left(\binom{d+\al_{ij}}4
+\binom{d-\al_{ij}}4 \right)}{\displaystyle\prod_{1\le i,j\le
3}(a_i+b_j)},
 \end{equation}
\end{prop}
Then obviously if $X=\wP \cap \left(\cap_{i=1}^2 (g_i)\right)$ is 
complete intersection Del Pezzo surface of index $I$ then $$-K_X^2=I^2 \prod_{i=1}^2\deg(g_i)\deg(
\wP)  $$


\subsection{Proof of Theorem \ref{thrm-pxp1}}
 We prove the  existence of each model with right invariants and singularities. Then we prove the  quasi-smoothness of the each model. We write the proof for a general member  $X$ of a family $\sX\subset \PP(a,b,c,d,e,f,g)$ of orbifold Del Pezzo surfaces. In the course of the proof our parameter is $r$ and the rest of the weights in terms of $r$ are  $$\begin{array}{llll}q=r-1,& s=r+1,&t=r+2,&u=r+3,\\v=2r+1,&y=2r-2,&z=2r-1,&m=3r-2.\end{array}$$ 
\subsubsection{\textbf{$\PxP$ model 11}}
\label{S-P11}
 The parameter $r=n, n \in \NN$ and for an  input parameter $w=(0,0,q;1,1,r),$ we get the embedding of a 4 dimensional orbifold  $$\wP \into \PP(1^4,r^4,z) \text{ with }K_{\wP}=\Oh(-(2r+1)).$$  Then if we take the complete intersection of 2 forms of degree $r$ then
$$X:= \wP\cap (r)^2 \into \PP(1^4,r^2,z)=\PP(a,b,c,d,e,f,g)$$ is a  Del Pezzo surfaces of index 1. \\
The equations of $X$ can be described by the $2 \times 2$ minors of  

$$\left[\begin{matrix}
a_1 &b_1&c_r\\
d_1&e_1&H_r\\
f_r& J_r&g_z
\end{matrix}\right].
$$
\textit{Quasismoothness:}
Any member of the family of  Del Pezzo orbifolds $ X$ has two non-trivial orbifold strata, one is of weight $z$ and the other of weight $r$. The weight $z$ locus is just a coordinate point which obviously lies on $X$. The variables $a, b,d$ and $e$ serves as tangent variables to the variable $g$. Thus it is a point of type $\frac 1{z}(1,1)$. On the other hand the weight $r$ locus is a copy restricted to $X$ is an empty set, so we do not get any other orbifolds singularities on $X$.  \\
 The  base locus of linear system $|\Oh(r)|$ consists of just a coordinate point of weight $z$ 
which is manifestly quasismooth. Thus the each member in this model  is a wellformed and quasismooth orbifold Del Pezzo surface with an orbifold point of type $\frac{1}{z}(1,1)$.    
\subsubsection{\textbf{$\PxP$ model 12}}
\label{S-P12} The paramter has the form $r=2n+1, n \in \NN$  and for an input parameter $w=(0,1,q;1,r,s)$, we have the embedding  $$
 \wP \into \PP(1,2,r^2,s^2,t,z,2r)
 \text{ and }K_{\wP}=\Oh(-(3r+2)).$$
Then a log  Del Pezzo surface of index 1 is obtained as the complete intersection of $\wP$ with a form degree  $2r$ and $s$ to get 
$$X:= \wP\cap (2r)\cap(s) \into \PP(1,2,r^2,s,t,z)=\PP( a,b,c,d,e,f,g).$$  Let $H_{s}$ and $H_{2r}$ denote the the weighted homogeneous forms of degree $s$ and $2r$ respectively, then the  equations  of $X$ are given  by the  $2\times 2$ of
$$\left[\begin{matrix} 
a_1&b_2&c_r\\ 
d_r&e_s&f_z\\ 
H_s&g_t&H_{2r}\end{matrix}\right].
$$ \\
\textit{Quasismoothness:}
The details of orbifold loci lying on $X$ is summarized as follows.

$$
\renewcommand*{\arraystretch}{1.2}
\begin{array}{cccl}
P_\orb & X\cap P_\orb \ & \text {Tangent}|\text{local variables}  & \text{Conditions on forms}  \\\hline
        z & \text {coordinate pt } f_z & a,b,g,e\;|\;c,d & H_s=e+\cdots \\
 t & \text {coordinate pt } g_t & a,d,c,f\;|\;b,e & F_3=d+\cdots \\
 r & \text {coordinate pt } c_r & d,e,g,a\;|\;b,f & H_s=ac+\cdots \\

        \end{array}
$$

 The weight $s$ locus do not intersect  $X$ as $H_s=e+\cdots$. The weight $2$ locus consists of $\PP[2,s]$ as $s$ is even but  its an  empty set as $H_{2r}=b^r+\cdots$. \\ Moreover, $X$ may contain the orbifold locus of weight 5 as $X\cap \PP[t,z]$, for example for $r=13$ but this does not give any new singularities at all. Similarly, $X$ may contain weight $3$ locus as $X\cap \PP[s,z]$; for example for $r=5$ but again this does not give any new singular point. Thus $X$ is a quasismooth on the orbifold locus. The quasismoothness on the base locus has been verified by computer algebra for $3\le n \le 99.$
    
\subsubsection{\textbf{$\PxP$ model 13}}
\label{S-P13}The paramter has the form $r=2n+1, n \in \NN$  and for an  input parameters $w=(0,1,q;1,2,r)$,   we get the following embedding of the 4-fold 
 $$
 \wP \into \PP(1,2^2,3,r^2,s^2,z) \text{ and }K_{\wP}=\Oh(-(2r+3)).$$
 The the complete intersection with two forms of degree $s$ is a  log  Del Pezzo surface 
$$X:= \wP\cap (s)^{2} \into \PP(1,2^2,3,r^2,z)=\PP( a,b,c,d,e,f,g)$$ of index 1.   The equations of  $X$ are    $2\times 2$ minors of  
\renewcommand*{\arraystretch}{1.2}
$$\left[\begin{matrix} 
a_1&b_2&c_r\\ 
d_2&e_3&H_s\\ 
f_r&J_s&g_z\end{matrix}\right].
$$ 
\textit{Quasismoothness:}
The details of orbifold loci lying on $X$ is summarized as follows.

$$
\renewcommand*{\arraystretch}{1.2}
\begin{array}{c|cc|l}
P_\orb & X\cap P_\orb \ & \text {Tangent}|\text{local variables}  & \text{Conditions on forms}  \\
        z & \text {coordinate pt } g_z & a,b,d,e\;|\;c,f &  \\
\evnrow         r&  \begin{matrix}
\text { 2 pts, on coordinate pt  } c\\
\text{ on coordinate pt } f

\end{matrix}   & \begin{matrix}
d,e,f,a\;|\;b,g \\
b,c,e,a\;|\;d,g
\end{matrix}&\begin{matrix}
J_s=ac+\cdots\\
H_s=af+\cdots\end{matrix} \\
3 & \text {coordinate pt } e_3 & a,c,f,g\;|\;b,d & 
        \end{array}
$$

 The weight $3$ is a coordinate point if $r$ is not divisible by 3. Otherwise, it defines 3 coordinate points of variables $c,e$ and $f$ and only the coordinate point of the variable $e$ gives a new orbifold point. if $z=0 \mod 3$ then it  gives only one new point.  The weight $2$ locus consists of $X\cap \PP[b,d]$, which does not itesect $X$ as these variables also appear in $H_s$ and $J_s$. Thus $X$ is a quasismooth on the orbifold locus.  
The quasismoothness for this model has been verified  by computer algebra for $3\le r \le 37.$
\subsubsection{\textbf{$\PxP$ model 14}}
\label{S-P14} The paramter has the form $r=6n-1, n \in \NN$  and for     input parameters $w=(0,1,q;2,r,s)$ to get 
 $$
 \wP \into \PP(2,3,r,s^3,t,z,2r) \text{ with }
K_\wG=\Oh(-3(r+1))).$$
We take a projective  cone of weight $r$ to get a 5-fold $\mathcal{C}^r\wP$ with the canonical divisor class $\Oh(-(4r+3)).$ Then the complete intersection of this 5-fold with  two forms of  weight $s$ and one form of weight $2r$ is  a log  Del Pezzo surface 
$$X:= \mathcal{C}^3\wP\cap (s)^2\cap(2r) \into \PP(r,2,3,r,s,t,z)=\PP(g, a,b,c,d,e,f)$$ of  index 1.   The equations are given by $2\times 2$ minors of  
$$\left[\begin{matrix} 
a_2&b_3&c_s\\ 
d_r&H_s&e_z\\ 
J_s&f_t&H_{2r} \end{matrix}\right],
$$
where $H_s,J_s$ and $H_{2r}$ are forms of degree $s,s$ and $2r$ respectively.\\
\textit{Quasismoothness:}
The details of orbifold loci lying on $X$ is summarized as follows.

$$
\renewcommand*{\arraystretch}{1}
\begin{array}{c|cc|l}
P_\orb & X\cap P_\orb \ & \text {Tangent}|\text{local variables}  & \text{Conditions on forms}  \\  
  \evnrow z & \text {coordinate pt } e_z & a,b,f,c\;|\;d,g & G_s=c+\cdots \\
 \oddrow t & \text {coordinate pt } f_t & a,c,d,e\;|\;b,g & \\

 \evnrow r&  \begin{matrix}
\text { 3 points by \ref{lemma-fletcher} }\\
\text{ on patch } d\\
\text{ on patch }g

\end{matrix}   & \begin{matrix}
\\
b,c,f,g\;|\;a,e \\
a,b,d,c\;|\;e,f
\end{matrix}&\begin{matrix}\\
H_s=c+\cdots,H_{2r}=dg+\cdots\\
H_{2r}=g^2+\cdots,
H_s=c+\cdots
\end{matrix}  \\
     
\oddrow                         3 & \text {coordinate pt } b_3 & d,e,c,f\;|\;a,g & H_{2r}=b^{2n}f+\cdots,F_t,G_t=c+\cdots
            
        \end{array}
$$
The weight $s$ orbifold locus obviously does not intersect $X$. The weight 3 locus is  $X\cap \PP[3,s,z]$ if $z=0\mod 3$ and $X\cap \PP[3,s]$ otherwise. In both we only get the coordinate point of variable $b$ as a new orbifold point. Finally, the weight 2 locus is given by $X\cap \PP[2,s]$ which obviously an empty set.  Thus $X$ has the correct type of singularities  and quasismooth on the orbifold locus. The quasismoothness  has been verified by computer algebra for $5\le r \le 611 .$


\subsection{Fano index 2 models}
In this case we get  two bi-parameterized   models  of families of log  Del Pezzo surfaces with rigid singularities, i.e. indexed by \(\NN \times \NN \).  If we fix one parameter,  we get  a one  parameter model which we computed earlier. Thus each of these two parameter models can be considered as consisting of infinite series of one parameter models.\\[2mm]
\textbf{Proof of Theorem \ref{thrm-pxp2}:}
Let $r=3m$ and \(y=3n+1\) be the two parameters with $q=r-1,s=r+1$ and $z=y+1$ .   If we choose the input parameter $(0,1,q;1,2,y)$ then we get the embedding of the ambient 4 dimensional orbifold $\wP$  

$$\wP\into\PP(1,2^2,3,r,s,y,z,w),$$
where $w=q+y.$
The orbifold canonical class,  computed from the Hilbert series, is given by  $$K_{\wP}=\Oh(-(r+y+3)).$$ 
Then the complete intersection with a form weight $w=r+y-1$ and a quadric gives a  Del Pezzo surface of index 2:
$$X=\wP\cap (H_w)\cap (H_2)\into \PP(1,2,r,3,s,y,z)=\PP(a,b,c,d,e,f,g).$$ 
The equations can be described by the $2\times 2$ minors of the following matrix.

\renewcommand*{\arraystretch}{1.2}
$$\left[\begin{matrix} 
a_1&b_2&c_r\\ 
H_2&d_3&e_s\\ 
f_y&g_z&H_w\end{matrix}\right].
$$  Now we prove the quasi smoothness of each model one by one by the analysis  of its orbifold loci.
\subsubsection{\textbf{$\PxP$ model 21}}
\label{S-P21}
In this model the parameters $r$ and $y$ both are multiples of 3. The details of orbifold loci lying on $X$ is summarized as follows.

$$
\renewcommand*{\arraystretch}{1.2}
\begin{array}{cccl}
P_\orb & X\cap P_\orb \ & \text {Tangent}|\text{local variables}  & \text{Conditions on forms}  \\\hline
 \evnrow z & \text {coordinate pt } g_z & a,c,e,b\;|\;d,f & H_2=b+\cdots  \\
  y & \text {coordinate pt } f_y & b,c,d,e\;|\;a,g &   \\
  \evnrow s & \text {coordinate pt } e_s & a,b,f,g\;|\;c,d &   \\
r & \text {coordinate pt } c_r & d,f,g,b\;|\;a,c & H_2=b+\cdots  \\
\evnrow 3 & \text {coordinate pt } d_3 & a,c,f,b\;|\;e,g & 
H_w=d^{m+n-1}b+\cdots        \end{array}
$$

The weight $3$ is given by $X\cap \PP[3,r,y]$ which gives 3 coordinate points. But only one new orbifold point appear on this locus as the coordinate points of weight $r$ and $y$ are counted earlier. The weight $2$ locus is $X\cap \PP[2,r,y]$ if $r \text{ and }y$ are even and is given by $X\cap \PP[2,s,z]$ otherwise. Essentially we do not get any new orbifold since in each case we get $bH_2$ as one of the defining equations of weight 2 locus.  Thus $X$ is a quasismooth on the orbifold locus.  

We verified the quasismoothness by using  computer algebra  for $6\le r, y\le 63. $ 
\subsubsection{\textbf{$\PxP$ model 22}}
\label{S-P22}
In this model we have $r=3m$, $y=3n+1$ such that $n\ge 2$ and $m\ge n$. The analysis of orbifold is similar for weight $z,y,s$ and $r$ to the firs model in the previous case. The weight 3 locus is given by  $X\cap \PP[3,r]$ which is  given by $$V(dH_w,cd)\subset \PP[c,d],$$ since 3 divides $w$. This gives a coordinate point $c$ which is already considered as $\frac 1r$ orbifold point. The weight 2 locus is $X\cap \PP[2,r,z]$ if $r$ is even and $X\cap \PP[2,s,y]$ otherwise. In both cases, we do not get any new orbifold point.
 Thus $X$ is a quasismooth on the orbifold locus. One can show that the base locus remains same for each value of parameter, following section \ref{S-Pf12}. We used the computer algebra to verify the quasismoothness  for $6\le r,y\le 63. $

\section{Summary of computational results and sporadic examples}
\label{Sporadic}
\subsection{Summary of computational results}We use  computer search routine of \cite{QJSC,BKZ} to search  families of orbifold Del Pezzo surface with isolated orbifold points.  The computer search is carried out in order of increasing the adjunction number of the Hilbert series of each format; separately for  each  Fano index. The computer search result do not have a termination condition but  it returns complete list of candidate families for each adjunction number. Thus our results are complete up to a certain value of adjunction number.   The Table \ref{tab-summary} contains the   summary of the computational  results and details of sporadic families of log Del Pezzo surfaces. It contains the number of candidates  from the computer search, how many of them contains only rigid singularities, how many of them are not in the models of the earlier sections  \S\ref{Pfaffians} and \S\ref{Binomials} and how many of those sporadic cases are quasismooth. 

\begin{table}[h]

\caption{ The first column contains the format and the Fano index of these orbifold Del Pezzo surface.  
 The column $q_{\mathrm{max}}$  gives the largest adjunction number
 searched in the given format;  \# \(X_c\) gives the number of candidates returned, \#\(X_{rc}\) contains the number of candidates with only rigid singularities, \#\(X_{src}\) contains the number sporadic rigid candidates (those candidates which are not in any models) and the last column \#QS-\(X_{src}\) gives the number of sporadic rigid families which are   quasismooth.
\label{tab-summary}}
\[
\renewcommand*{\arraystretch}{1.2} \begin{tabular}{|c|c|c|c|c|c|}\toprule\hline
 Format-\(I\) &  $q_{\max}$  & \# \(X_{c}\)
 & \# \(X_{rc}\)& \# \(X_{src}\) & \# QS--\(X_{src }  \)   
 \\\hline
\evnrow \(\wG\)-1 & 60 & 39 & 39 & 12 &10 \\\hline
\oddrow \(\wG\)-2 & 52 & 116 & 46 & 29 &6\\\hline
\evnrow \(\wF\)-1 & 80 & 49 & 49 & 14 &7\\\hline
\oddrow \(\wF\)-2 & 68 & 127 & 44 & 5 &0\\\hline

\end{tabular}\]
\end{table}

It is important to mention that  there are other type of key varieties appearing in \cite{wg,QS,QS2,QJGP} in codimension $c=5,\ldots,10$ which can be used as ambient varieties to search for  log Del Pezzo surfaces with isolated orbifold points. Indeed, we searched for the families of orbifold Del Pezzo surfaces in those formats but  the computer search do not provide even sporadic examples. 
      
\subsection{Sporadic cases}
\label{S-sporadic}
In this section we present  those cases which are not in any of our models but appear as sporadic families of orbifold Del Pezzo surfaces in each format.   We list the families of wellformed and quasismooth Del Pezzo surfaces with rigid singularities whos equations are  those given   by either the maximal  Pfaffians of $5\times 5$ skew symmetric matrix or by the  $2\times 2$ minors of order 3  matrices. Indeed, we follows all the steps of Section \ref{proof-strategy} to prove the  existence, wellformedness and quasismoothness of the following families of orbifold Del Pezzo surfaces.  In particular, we used  computer algebra calculations to show the quasismoothness on base locus in these examples. We summarized the results in the form of the following table.

\begin{longtable}{>{\hspace{0.5em}}llccccr<{\hspace{0.5em}}}
\caption{ log Del Pezzo surfaces in Pfaffian format} \label{tab!codim3}\\
\toprule
\multicolumn{1}{c}{WPS \&\ Para}&\multicolumn{1}{c}{Basket $\mathcal{B}$}&\multicolumn{1}{c}{$-K^2_X$}&\multicolumn{1}{c}{$h^0(-K)$}&\multicolumn{1}{c}{$I$}&\multicolumn{1}{c}{Weight Matrix}\\
\cmidrule(lr){1-1}\cmidrule(lr){2-2}\cmidrule(lr){3-5}\cmidrule(lr){6-6}
\endfirsthead
\multicolumn{7}{l}{\vspace{-0.25em}\scriptsize\emph{\tablename\ \thetable{} continued from previous page}}\\
\midrule
\endhead
\multicolumn{7}{r}{\scriptsize\emph{Continued on next page}}\\
\endfoot
\bottomrule
\endlastfoot

\evnrow $\begin{array}{@{}l@{}}\evnrow \PP(1,3^2,5^2,7),\\ w=\frac 12(-1, 3, 3, 7, 11)\end{array}$& $\dfrac 13(1,1),  \dfrac 15(1,1), \dfrac 17 (1,4)$&$\dfrac{29}{105}$&$1$&$1$&$\footnotesize\begin{matrix} 1&1&3&5\\ &3&5&7\\ &&5&7 \\ &&&9 \end{matrix}$\\
\oddrow $\begin{array}{@{}l@{}}\oddrow \PP(3^2,5^2,7^2),\\ w=\frac 12(1, 5, 5, 9, 9)\end{array}$& $ 3\times \dfrac 13(1,1),\dfrac 15(1,1),2\times \dfrac 17(1,4)$&$\dfrac{3}{35}$&$0$&$1$&$\footnotesize\begin{matrix} 3&3&5&5\\ &5&7&7\\ &&7&7\\ &&&9 \end{matrix}$\\
\evnrow $\begin{array}{@{}l@{}}\evnrow \PP(1,3,5,7,9,11),\\ w=\frac 12(-1, 3, 7, 11, 15)\end{array}$& $\dfrac 13(1,1),  \dfrac 1{11}(1,3) $&$\dfrac{5}{33}$&$1$&$1$&$\footnotesize\begin{matrix} 1&3&5&7\\ &5&7&9\\ &&9&11 \\ &&&13 \end{matrix}$\\
\oddrow $\begin{array}{@{}l@{}}\oddrow \PP(3,5,6,7,8,13),\\ w= \frac 12(1, 5,  9,11,15)\end{array}$& $ 2\times \dfrac 13(1,1), \dfrac 17(1,4),\dfrac{1}{13}(1,9)$&$\dfrac{11}{273}$&$0$&$1$&$\footnotesize\begin{matrix} 3&5&6&8\\ &7&8&10\\ &&10&12\\ &&&13 \end{matrix}$\\
\evnrow $\begin{array}{@{}l@{}}\evnrow \PP(1,4,5,7,11,17),\\ w=(0,1,4,7,10)\end{array}$& $\dfrac 1{17}(1,4) $&$\dfrac{2}{17}$&$1$&$1$&$\footnotesize\begin{matrix} 1&4&7&10\\ &5&8&11\\ &&11&14 \\ &&&17 \end{matrix}$\\

\oddrow $\begin{array}{@{}l@{}}\oddrow \PP(3,5,7,9,11,13),\\ w=\frac 12( 3, 7, 11, 11, 15)\end{array}$& $\dfrac 13(1,1),  \dfrac 1{5}(1,2) \dfrac 17 (1,4),\dfrac{1}{13}(1,8) $&$\dfrac{41}{1365}$&$0$&$1$&$\footnotesize\begin{matrix} 5&7&7&9\\ &7&9&11\\ &&11&13 \\ &&&15 \end{matrix}$\\
\evnrow $\begin{array}{@{}l@{}}\evnrow \PP(4,5,7,10,11,13),\\ w= \frac 12(1, 7,  9,13,19)\end{array}$& $  \dfrac 15(1,2),\dfrac{1}{11}(1,8),\dfrac{1}{13}(1,9)$&$\dfrac{16}{715}$&$0$&$1$&$\footnotesize\begin{matrix} 4&5&7&10\\ &8&10&13\\ &&11&14\\ &&&16 \end{matrix}$\\
\oddrow $\begin{array}{@{}l@{}}\oddrow \PP(1,3,7,8,13,19),\\ w= (-2,3,5,9,10)\end{array}$& $  \dfrac 17(1,4),\dfrac{1}{13}(1,9),\dfrac{1}{19}(1,10)$&$\dfrac{191}{1729}$&$0$&$1$&$\footnotesize\begin{matrix} 1&3&7&8\\ &8&12&13\\ &&14&15\\ &&&19 \end{matrix}$\\
\evnrow $\begin{array}{@{}l@{}}\evnrow \PP(1,5,6,9,14,19),\\ w= \frac 12(-3,5,13,15,23)\end{array}$& $  \dfrac 19(1,1),\dfrac{1}{19}(1,15)$&$\dfrac{13}{171}$&$1$&$1$&$\footnotesize\begin{matrix} 1&5&6&10\\ &9&10&14\\ &&14&18\\ &&&19 \end{matrix}$\\
\oddrow $\begin{array}{@{}l@{}}\oddrow \PP(1,5,7,10,14,23),\\ w=\frac 12 (-1,3,11,17,19)\end{array}$& $  \dfrac 15 (1,2),\dfrac 17(1,4),\dfrac{1}{23}(1,6)$&$\dfrac{52}{805}$&$1$&$1$&$\footnotesize\begin{matrix} 1&5&8&14\\ &7&10&16\\ &&14&20\\ &&&23 \end{matrix}$\\
\evnrow $\begin{array}{@{}l@{}}\evnrow \PP(1,3,5,6,7,8),\\ w=(0,1,2,5,6)\end{array}$& $\dfrac 13(1,1),  \dfrac 15(1,1), \dfrac 18 (1,5)$&$\dfrac{19}{30}$&$1$&$2$&$\footnotesize\begin{matrix} 1&2&5&6\\ &3&6&7\\ &&7&8 \\ &&&11 \end{matrix}$\\
\oddrow $\begin{array}{@{}l@{}}\oddrow \PP(4,5^2,7^2,8),\\ w=(2,2,3,5,5)\end{array}$& $ 2\times \dfrac 15(1,3),2\times \dfrac 17(1,3),\dfrac 18(1,1)$&$\dfrac{11}{70}$&$0$&$2$&$\footnotesize\begin{matrix} 4&5&7&7\\ &5&7&7\\ &&8&8\\ &&&10 \end{matrix}$\\
\evnrow $\begin{array}{@{}l@{}}\evnrow \PP(3,6,7^2,8^2),\\ w=\frac12(5,7,7,9,9)\end{array}$& $\dfrac 13(1,1),  \dfrac 16(1,1), \dfrac{1}{7}(1,2),2\times \dfrac 18 (1,5)$&$\dfrac{1}{7}$&$0$&$2$&$\footnotesize\begin{matrix} 6&6&7&7\\ &7&8&8\\ &&8&8 \\ &&&9 \end{matrix}$\\
\oddrow $\begin{array}{@{}l@{}}\oddrow \PP(3,6,7^2,8,11),\\ w=(1,2,5,6,6)\end{array}$& $ 2\times \dfrac 13(1,1), \dfrac 17(1,2),\dfrac 18(1,5), \dfrac{1}{11}(1,3)$&$\dfrac{59}{462}$&$0$&$2$&$\footnotesize\begin{matrix} 3&6&7&7\\ &7&8&8\\ &&11&11\\ &&&12 \end{matrix}$\\
\evnrow $\begin{array}{@{}l@{}}\evnrow \PP(3,7,8^2,9,10),\\ w=\frac12(7,7,9,9,11)\end{array}$& $\dfrac 13(1,1),  \dfrac 17(1,1), 2\times \dfrac 18 (1,5),\dfrac{1}{10}(1,3)$&$\dfrac{11}{105}$&$0$&$2$&$\footnotesize\begin{matrix} 7&8&8&9\\ &8&8&9\\ &&9&10 \\ &&&10 \end{matrix}$\\
\oddrow $\begin{array}{@{}l@{}}\oddrow \PP(4,7^2,9^2,12),\\ w=(2,2,5,7,7)\end{array}$& $ 2\times \dfrac 17(1,3), 2\times \dfrac 19(1,4), \dfrac{1}{12}(1,1)$&$\dfrac{5}{63}$&$0$&$2$&$\footnotesize\begin{matrix} 4&7&9&9\\ &7&9&9\\ &&12&12\\ &&&14 \end{matrix}$\\

\end{longtable}

\begin{longtable}{>{\hspace{0.5em}}llccccr<{\hspace{0.5em}}}
\caption{log Del Pezzo surfaces in $\PxP$ format} \label{tab!codim4}\\
\toprule
\multicolumn{1}{c}{WPS, Para}&\multicolumn{1}{c}{Basket $\mathcal{B}$}&\multicolumn{1}{c}{$-K^2_X$}&\multicolumn{1}{c}{$h^0(-K)$}&\multicolumn{1}{c}{$I$}&\multicolumn{1}{l}{Weight Matrix}\\
\cmidrule(lr){1-1}\cmidrule(lr){2-2}\cmidrule(lr){3-5}\cmidrule(lr){6-6}
\endfirsthead
\multicolumn{7}{l}{\vspace{-0.25em}\scriptsize\emph{\tablename\ \thetable{} continued from previous page}}\\
\midrule
\endhead
\multicolumn{7}{r}{\scriptsize\emph{Continued on next page}}\\
\endfoot
\bottomrule
\endlastfoot
\evnrow $\begin{array}{@{}l@{}}\evnrow \PP(1,2,3^2,4,5^2),\\ w=(0,1,2;1,3,4)\end{array}$& $\dfrac 13(1,1), \dfrac 15(1,2), \dfrac 15(1,1)$&$\dfrac{8}{15}$&$1$&$1$&$\footnotesize\begin{matrix} 1&2&3\\ 3&4&5\\ 4&5&6 \end{matrix}$\\
\oddrow $\begin{array}{@{}l@{}}\oddrow \PP(3,4,5^2,6,7^2),\\ w=(0,1,2;4,5,6)\end{array}$& $ 2\times \dfrac 15(1,2),2\times \dfrac 17(1,4)$&$\dfrac{3}{35}$&$0$&$1$&$\footnotesize\begin{matrix} 4&5&6\\ 5&6&7\\ 6&7&8 \end{matrix}$\\
\evnrow $\begin{array}{@{}l@{}}\evnrow \PP(2,3,5^2,6,9,13)\\ w=(0,1,4;2,5,9)\end{array}$& $2\times\dfrac 13(1,1),2\times \dfrac 15(1,2), \dfrac 15(1,1),\dfrac 1 {13}(1,7)$&$\dfrac{22}{195}$&$0$&$1$&$\footnotesize\begin{matrix} 2&3&6\\ 5&6&9\\ 9&10&13 \end{matrix}$\\
\oddrow $\begin{array}{@{}l@{}}\oddrow \PP(1,3,5^2,7,9,13)\\ w=(0,2,4;1,5,9)\end{array}$&
$ \dfrac 15(1,2),\dfrac 15(1,1), \dfrac 17 (1,4),\dfrac{1}{13}(1,7)$&$\dfrac{86}{455}$&$1$&$1$&$\footnotesize\begin{matrix}
1&5&9\\ 3&7&11\\ 5&9&13 \end{matrix}$\\

\evnrow $\begin{array}{@{}l@{}}\evnrow \PP(1,5^2,7,9,11,17)\\ w=(0,4,6;1,5,11)\end{array}$& $ \dfrac 15(1,2), \dfrac 17(1,4), \dfrac 19 (1,1),\dfrac{1}{17}(1,4)$&$\dfrac{562}{5355}$&$1$&$1$&$\footnotesize\begin{matrix} 1&5&7\\ 5&9&11\\ 11&15&17 \end{matrix}$\\
\oddrow $\begin{array}{@{}l@{}}\oddrow \PP(2,3,7^2,8,13,19)\\ w=(0,1,6;2,7,13)\end{array}$& $\dfrac 13(1,1),2\times \dfrac 17(1,3), \dfrac 17(1,4),\dfrac 1 {19}(1,12)$&$\dfrac{4}{57}$&$0$&$1$&$\footnotesize\begin{matrix} 2&3&8\\ 7&8&13\\ 13&14&19 \end{matrix}$\\

\end{longtable}

\bibliographystyle{amsalpha}
\bibliography{References}


\end{document}